\documentclass[11pt]{amsart}
\usepackage{amsmath}
\usepackage{amssymb}
\usepackage{amscd}
\usepackage{color}

\usepackage{tikz}
\usetikzlibrary{matrix}
\usetikzlibrary{arrows,calc}
\allowdisplaybreaks
\def\NZQ{\mathbb}               
\def\NN{{\NZQ N}}

\def\ZZ{{\NZQ Z}}
\def\RR{{\NZQ R}}

\newtheorem{Theorem}{Theorem}[section]
\newtheorem{Lemma}[Theorem]{Lemma}

\newtheorem{Remark}[Theorem]{Remark}

\let\epsilon\varepsilon
\let\phi=\varphi
\let\kappa=\varkappa

\begin{document}
\title{Limits of Amao multiplicities and epsilon multiplicity}

\author{Steven Dale Cutkosky}

\thanks{Partially supported by NSF grant DMS-2348849.}

\address{Steven Dale Cutkosky, Department of Mathematics,
University of Missouri, Columbia, MO 65211, USA}
\email{cutkoskys@missouri.edu}

\begin{abstract}
Let $R$ be a $d$-dimensional Noetherian local ring with maximal ideal $m_R$ and let $I$ be an ideal of $R$. The epsilon multiplicity $\epsilon(I)$ of $I$ is 
defined by Ulrich and Validashti in \cite{UV} as a  limsup of normalized lengths. We show in this paper that the epsilon multiplicity is always a limit of these normalized lengths.
Specifically,
$$
\epsilon(I)=
\lim_{n\rightarrow\infty} \frac{\ell_R((I^n)^{\rm sat}/I^n)}{n^d/d!}.
$$
 We prove that the epsilon multiplicity satisfies an associativity formula, analogous to that of the ordinary multiplicity. 
 
 We prove these results by comparing the epsilon multiplicity to a suitable limsup of normalized Amao multiplicities, the Amao epsilon multiplicity $a\epsilon(I)$. We show that the epsilon multiplicity is always equal to the Amao epsilon multiplicity, which is itself a limit.
\end{abstract}


\maketitle

\section{Introduction} 

Epsilon multiplicity is defined for ideals, modules and graded algebras by Bernd Ulrich and Javid Validashti in their paper \cite{UV}. A motivation is to find criteria for 
 equisingularity in families of arbitrary singularities, extending work of Teissier, Gaffney, Gassler,  Kleiman, Thorup  and others
 \cite{T1}, \cite{T2}, \cite{GG}, \cite{KT}.

 In their paper, Ulrich and Validashti  define the epsilon multiplicity as a limsup, which is always finite since it is bounded from above by the $j$-multiplicity.
 They show that if $R$ is a Noetherian local ring of dimension $d$ which is equidimensional and universally catenary,  and $U\subset E$ are rank $e$ submodules of a free module of finite rank $e$, then  $E$ is integral over $U$ if and only if there is equality of epsilon multiplicities 
 $\epsilon(U_q)=\epsilon(E_q)$ for all prime ideals $q$ of $R$.

Let $R$ be a $d$-dimensional Noetherian local ring with maximal ideal $m_R$ and $I\subset R$ be an ideal. The epsilon multiplicity of $I$ is defined in \cite{UV} as 
$$
\epsilon(I)=
\limsup_{n\rightarrow\infty} \frac{\ell_R((I^n)^{\rm sat}/I^n)}{n^d/d!},
$$
where $\ell_R(M)$ is the length of an $R$-module $M$.
The saturation $J^{\rm sat}$ of an ideal $J$ of $R$ is 
$$
J^{\rm sat}=J:m_R^{\infty}=\cup_{i=0}^{\infty}J:m_R^i
$$
where $m_R$ is the maximal ideal of $R$.
It is shown in Corollary 6.3 \cite{C2} that if $R$ is analytically unramified, then the epsilon multiplicity of an ideal in $R$ exists as a limit.

The epsilon multiplicity can be an irrational number, as shown in \cite{CHST}.

Some recent papers on epsilon multiplicity are Suprajo Das \cite{sD}, Cutkosky and Sarkar \cite{CS}, Suprajo Das, S. Dubey, S. Roy and J. Verma \cite{DDRV}, Suprajo Das, S. Roy and V. Trivedi \cite{DRT} and \cite{DRT2}, Cutkosky and Landsittel \cite{CL} and Landsittel \cite{L}.

An important question in this theory is if the epsilon multiplicity is actually a limit, and not just a limsup. We show in this paper that this question always has a positive answer.

\begin{Theorem}\label{TheoremA5} Let $R$ be a $d$-dimensional Noetherian local ring and $I$ be an ideal of $R$. Then 
$$
\epsilon(I)=\lim_{n\rightarrow\infty}\frac{\ell_R((I^n)^{\rm sat}/I^n)}{n^d/d!}
$$
exists as a limit.
\end{Theorem}

We  obtain the following associativity formula for epsilon multiplicity.

\begin{Theorem}\label{TheoremA6} Let $R$ be a $d$-dimensional Noetherian local ring and $I\subset R$ be an ideal.
Let $P_1,\ldots ,P_r$ be the minimal prime ideals of $R$ such that $\dim R/P_i=d$. Then
\begin{equation}
\epsilon(I)=\sum_{i=1}^r\ell_{R_{P_i}}(R_{P_i})\epsilon(I(R/P_i)).
\end{equation}
\end{Theorem}

Theorem \ref{TheoremA6} follows from Theorems \ref{TheoremB2} and \ref{TheoremA4}.

The proof of Theorem \ref{TheoremA5}  is accomplished by defining another multiplicity, $a\epsilon(I)$, the Amao epsilon multiplicity, as a limsup of normalized Amao multiplicities. 
The  Amao multiplicity is defined and discussed below. 
We show that $a\epsilon(I)$ always exists as a limit in Theorem \ref{TheoremA1}. To prove Theorem \ref{TheoremA5},
we use the fact 
shown in \cite{CL} with Stephen Landsittel, and stated in Theorem \ref{TheoremA} in Section \ref{PrelimSec}, that if $I$ is a ideal in an analytically unramified local ring $R$ then $\epsilon(I) =a\epsilon(I)$.  
We  show that for general $R$, $a\epsilon(I)$ satisfies an associativity formula in Theorem \ref{TheoremB2} and show that  the corresponding formula is an inequality for $\epsilon(I)$ in Theorem \ref{TheoremA3}. Then using the fact that $a\epsilon(I)\le\epsilon(I)$ we deduce that $\epsilon(I)=a\epsilon(I)$ in Theorem \ref{TheoremA4}. We then further deduce that $\epsilon(I)$ is itself a limit in  Theorem \ref{TheoremA5}.

The  Amao multiplicity is defined and developed in \cite{Am}, \cite{R}, \cite{KR} and on page 332 of \cite{HS}. Suppose that $R$ is a $d$-dimensional Noetherian local ring and $K\subset J$ are ideals in $R$ such that $J/K$ has finite length. Then $J^n/K^n$ has finite length for all $n\ge 0$ and there exists a numerical polynomial $p(n)$ such that for $n$ sufficiently large, $p(n)=\ell_R(J^n/K^n)$. The Amao multiplicity $a(K,J)$ is defined to be $d!$ times the coefficient of the degree $d$ term of $p(n)$. That is,
$$
a(K,J)=\lim_{n\rightarrow\infty}\frac{\ell_R(J^n/K^n)}{n^d/d!}.
$$
The Amao multiplicity is always a nonnegative integer.

For $I$ an ideal in a $d$-dimensional Noetherian local ring $R$, let the Amao epsilon multiplicity of $I$ be
\begin{equation}\label{eqI1}
a\epsilon(I)=\limsup_{n\rightarrow\infty}\frac{a(I^n,(I^n)^{\rm sat})}{n^d}.
\end{equation}

In Theorem \ref{TheoremA4}, we show that the epsilon multiplicity and Amao epsilon multiplicity are always equal.

\begin{Theorem}\label{TheoremA4}
Suppose that  $R$ is a $d$-dimensional Noetherian local ring and $I\subset R$ is an ideal. Then 
$\epsilon(I)=a\epsilon(I)$.
\end{Theorem}

In  Theorem 1.5 \cite{L}, Landsittel proves that the epsilon multiplicity $\epsilon(I)$ of an ideal in a Noetherian $d$-dimensional local ring is equal to $a\epsilon(I)$ if the dimension of the nilradical of $\hat R$ is less than the dimension of $R$. The proof of  Theorem 1.5 in \cite{L} is by reduction to the case that $R$ is analytically unramified, which is proven in Theorem 1.1 \cite{CL} using Okounkov body methods. We will use this last result in our proof of Theorem \ref{TheoremA4} so we restate it as Theorem \ref{TheoremA} in Section \ref{PrelimSec}.

Theorem \ref{TheoremA}, Theorem 1.5 \cite{CL} and most generally Theorem \ref{TheoremA4},
 are generalizations of  the volume equals multiplicity formula of ordinary multiplicity to epsilon multiplicity. The volume = multiplicity formula of ordinary multiplicity is  proven in increasing generality in Ein, Lazarsfeld and Smith \cite{ELS}, Musta\c{t}\u{a} \cite{M1}, Lazarsfeld and Musta\c{t}\u{a} \cite{LM} and  Cutkosky, Theorem 6.5 \cite{C2} and Theorem 1.4 \cite{C4}. In contrast to the case of epsilon multiplicity, there are examples of graded families of $m_
 R$-primary ideals $\mathcal I$ on a $d$-dimensional Noetherian local ring $R$ for which the dimension of the nilradical of $\hat R$ has dimension $d$, such that the volume of $\mathcal I$ is strictly less than the multiplicity of $\mathcal I$ (\cite{C5}). The multiplicity of $\mathcal I$ exists as a limit on all Noetherian local rings (\cite{C5}), but if the dimension of the nilradical of $\hat R$ is $d$, then there exist graded families of $m_R$-primary ideals on $R$ for which the volume is not a limit (\cite{C2}).

  Recently, volume = multiplicity formulas have been proven for $p$-families and weakly graded families, by Sudipta Das \cite{SD}, Sudipta Das and Cheng Meng \cite{DM} and Sarkar \cite{PS}.

 \vskip .2truein
 
 
  In the final section \ref{AmaoSegre}, we give a geometric interpretation of the Amao multiplicity in a local ring which is essentially of finite type over a field.
It follows from Theorem 5.7 \cite{KT} that Amao multiplicities $a(I,J)$ can be computed using intersection theory, specifically as a sum of suitable Segre products obtained by blowing up $I$ and $J$. This expression of $a(I,J)$ is given in Theorem \ref{Theorem1} of Section \ref{AmaoSegre}.

The Segre numbers determining $a(I^n,(I^n)^{\rm sat})$ are intersection products of Cartier divisors $F$ and $E_n$ obtained by blowing up $I$ and $(I^n)^{\rm sat}$. 
In Section \ref{SecComp}, we use the fact that $\{(I^n)^{\rm sat}\}$ is a graded family of ideals, giving the analogous statement about Cartier divisors that
$E_m+E_n\ge E_{m+n}$. We use this fact to show that the limit of the Segre products of $nF$ and $E_n$ divided by $n^d$ have a limit in Section \ref{SecComp}, from which a geometric interpretation of $a\epsilon(I)$, expressing it as a sum of limits of Segre numbers,  is deduced in Theorem \ref{Theorem3} of Section \ref{AmaoSegre}.  

\section{Preliminary results}\label{PrelimSec}

Let $R$ be a Noetherian local ring of dimension $d$. We will write 
$\mbox{Min}_d(R)$ for the set of minimal prime ideals of $R$ such that $\dim R/P=d$. A graded family of ideals of $R$ is a family $\mathcal I=\{I_n\}_{n\in \NN}$ of ideals $I_n$ of $R$ such that $I_0=R$ and $I_mI_n\subset I_{m+n}$ for all $m,n\in \NN$. The family is a filtration if $I_n\subset I_m$ whenever $n\ge m$.

We will use the following theorem. The proof uses Okounkov body techniques.
\begin{Theorem}\label{TheoremCL}(Theorem 4.1 \cite{CL})
    Suppose that $R$ is an analytically unramified Noetherian local ring of dimension $d>0$ and that $\mathcal{I} = \{I_n\}$ and $\mathcal{J} = \{J_n\}$ are graded families of ideals in $R$ such that $I_n\subset J_n$ for $n\geq 1$. Suppose that for $m\in\mathbb{Z}_{>0}$, there exist graded families of ideals $\mathcal{I}(m) = \{I(m)_i\}$ and $\mathcal{J}(m) = \{J(m)_i\}$ such that $I(m)_i\subset J(m)_i$, $I(m)_1=I_m$ and $I(m)_i\subset I_{mi}$ for all $i$, $J(m)_1=J_m$, and $J(m)_i\subset J_{mi}$ for all $i$. 
    Suppose that $I_i\subset P$ for all $i$ if $P$ is a minimal prime of $R$ such that $\dim(R/P)=d$ and $I_1\subset P$. Suppose further that there exists $c\in\mathbb{Z}_{>0}$ such that
\begin{equation*}
    I(m)_i\cap m_R^{cim} = J(m)_i\cap m_R^{cim}
\end{equation*} for all $i$ and $m$. Then the limit
\begin{equation*}
    F(\mathcal{I},\mathcal{J}) = \lim_{n\to\infty}\frac{\ell_R(J_n/I_n)}{n^d/d!}
\end{equation*}exists, and for $m\in\mathbb{Z}_{>0}$, the limits
\begin{equation*}
    G(\mathcal{I}(m),\mathcal{J}(m)) = \lim_{k\to\infty}\frac{\ell_R(J(m)_k/I(m)_k)}{k^d/d!}
\end{equation*}exist. Further, the limit
\begin{equation*}
    \lim_{m\to\infty}\frac{G(\mathcal{I}(m),\mathcal{J}(m))}{m^d}
\end{equation*}exists and,
\begin{equation*}
    F(\mathcal{I},\mathcal{J}) = \lim_{m\to\infty}\frac{G(\mathcal{I}(m),\mathcal{J}(m))}{m^d}.
\end{equation*}
\end{Theorem}

The condition in Theorem \ref{TheoremCL} on minimal primes is automatically satisfied in the case that $\{I_n\}$ is a filtration. 

Theorem \ref{TheoremCL} has the following consequence, which we will use in the proof of Theorem \ref{TheoremA4}.

\begin{Theorem}\label{TheoremA}(Theorem 1.1 \cite{CL})  Let $R$ be a  $d$-dimensional Noetherian local ring and let $I\subset R$ be an ideal. 
Suppose that $R$ is analytically unramified. Then
$$
\epsilon(I)=
\lim_{m\rightarrow\infty}\frac{a(I^m,(I^m)^{\rm sat})}{m^d}=a\epsilon(I).
$$
\end{Theorem}

The following theorem, which we will use in the proofs of Theorems  \ref{TheoremA1} and \ref{TheoremB2},  is proven by Kirby and Rees in \cite{KR}.

\begin{Theorem}\label{TheoremA2}(Theorem 6.3 (iii) \cite{KR}) Suppose that $(R,m_R)$ is a Noetherian local ring of dimension $d$ and $K\subset J$ are ideals of $R$ such that the length $\ell_R(J/K)<\infty$.
Let $P_1,\ldots, P_r$ be the minimal primes of $R$ such that $\dim R/P_i=d$. Then
\begin{equation}\label{eqa24}
a(K,J)=\sum_{i=1}^r\ell_{R_{P_i}}(R_{P_i})a(K(R/P_i),J(R/P_i)).
\end{equation}
\end{Theorem}

\section{Proofs of the main theorems}\label{MainSec}

\begin{Theorem}\label{TheoremA1}
Let $R$ be a $d$-dimensional Noetherian local ring  and $I\subset R$ be an ideal. Then 
\begin{equation}\label{eqa10}
a\epsilon(I)=\lim_{n\rightarrow\infty} \frac{a(I^n,(I^n)^{\rm sat})}{n^d}\in \RR
\end{equation}
exists as a limit. 
\end{Theorem}

\begin{proof}
First suppose that $d=0$. Then $R$ is Artinian and $m_R^N=0$ for some $N$. Thus $(I^n)^{\rm sat}=R$ for every $n$. If $I=R$ then all Amao multiplicities $a(I^n,(I^n)^{\rm sat})$ are zero. If $I\ne R$, then $I$ is nilpotent and for  $n\gg0$, $a(I^n,R)=\ell_R(R)$, so 
$$
\lim_{n\rightarrow\infty}a_R(I^n,(I^n)^{\rm sat})=\ell_R(R),
$$
establishing (\ref{eqa10}) when $d=0$.
We  may thus assume that $d>0$.

 Let $J_n=(I^n)^{\rm sat}$. The family $\{J_n\}$ is a graded filtration of ideals. 
We have that 
\begin{equation}\label{eqa}
J_n\hat R=((I\hat R)^n)^{\rm sat}.
\end{equation}
Our proof of (\ref{eqa}) is from the proof of Corollary 2.13 \cite{L}.  There exists $t>0$ such that $(I^n)^{\rm sat}=I^n:m_R^t$ and $(I^n\hat R)^{\rm sat}=I^n\hat R:m_{\hat R}^t$. Flatness of $R\rightarrow \hat R$ implies
$$
((I\hat R)^n)^{\rm sat}=I^n\hat R:m_R^t\hat R=(I^n:m_R^t)\hat R=(I^n)^{\rm sat}\hat R.
$$
For every $q\ge 1$,
$$
(J_n\hat R)^q/(I^n\hat R)^q\cong (J_n^q/I^{nq})\otimes_R\hat R\cong J_n^q/I^{nq}.
$$
since $J_n^q/I^{nq}$ has finite length. Thus
\begin{equation}\label{eqa1}
a(I^n,J_n)=a((I\hat R)^n,((I\hat R)^n)^{\rm sat}).
\end{equation}
We may thus assume that $R$ is complete. Let $P_1,\ldots,P_r$ be the minimal primes of $R$ such that $\dim R/P_i=d$. Let $A_i=R/P_i$, $\overline I_i=IA_i$, $L_{i,n}=J_nA_i$. Theorem \ref{TheoremA2} implies
\begin{equation}\label{eqa3}
a(I^n,J_n)=\sum_{i=1}^r\ell_{R_{P_i}}(R_{P_i})a(\overline I_i^n,L_{i,n}).
\end{equation}
(In general, $L_{i,n}\ne (\overline I_i^n)^{\rm sat})$. It remains to prove that
\begin{equation}\label{eqa4}
\lim_{n\rightarrow\infty}\frac{a(\overline I_i^n,L_{i,n})}{n^d}\mbox{ exists}
\end{equation}
for all $i$.

Fix $i$ and write $A=R/P_i$, $\overline I=IA$ and $L_n=J_nA$. Consider the graded families $\mathcal A=\{\overline I^n\}_{n\ge 0}$ and $\mathcal B=\{L_n\}_{n\ge 0}$. We have that $\overline I^n\subset L_n$ for all $n$ and the quotient $L_n/\overline I^n$ has finite length since it is a quotient of $J_n/I^n$.  
 
 For every $m\ge 1$, consider the graded families $\mathcal A(m)$ and $\mathcal B(m)$ where $\mathcal A(m)_q=\overline I^{mq}$ and $\mathcal B(m)_q=L_m^q$. Since $\mathcal B$ is graded, $L_m^q\subset L_{mq}$. 
 
 If $x\in J_m$, then for some $s$, $m_R^sx\subset I^m$.  Thus $m_A^s\overline x\subset \overline I^m$ where $\overline x$ is the class of $x$ in $A$, so 
 $\overline x\in (\overline I^m)^{\rm sat}$. Thus 
 \begin{equation}\label{eqa6}
 L_m\subset (\overline I^m)^{\rm sat}
 \end{equation}
 for all $m$.

Lemma 5.2 \cite{CL}  implies that there exists $c>0$ such that
\begin{equation}\label{eqa7}
\overline I^{mq}\cap m_A^{cmq}=((\overline I^m)^{\rm sat})^q\cap m_A^{cmq}
\end{equation}
for all $m,q\ge 1$. Now (\ref{eqa6}) implies $\overline I^{mq}\subset L_m^q\subset ((\overline I^m)^{\rm sat})^q$. Intersecting with $m_A^{cmq}$ and using equation (\ref{eqa7}) we obtain that 
\begin{equation}\label{eqa8}
\overline I^{mq}\cap m_A^{cmq}=L_m^q\cap m_A^{cmq}
\end{equation}
for all $m,q$. The minimal prime condition of Theorem \ref{TheoremCL} is satisfied since $\{\overline I^n\}$ is a graded filtration.
Thus the conditions of Theorem \ref{TheoremCL} are satisfied, and so the limit
\begin{equation}\label{eqa9}
\lim_{m\rightarrow\infty}\frac{1}{m^d}\left(\lim_{q\rightarrow \infty}\frac{\ell_A(L_m^q/\overline I^{mq})}{q^d/d!}\right)
\end{equation}
exists. The inner limit is the Amao multiplicity $a(\overline I^m,L_m)$ establishing (\ref{eqa4}). Finally, dividing (\ref{eqa3}) by $n^d$ and taking the limit as $n$ goes to infinity to conclude that 
$$
\lim_{n\rightarrow\infty}\frac{a(I^n,J_n)}{n^d}\mbox{ exists},
$$
establishing (\ref{eqa10}) when $d>0$.

\end{proof}

\begin{Theorem}\label{TheoremB2} Let $R$ be a $d$-dimensional Noetherian local ring and $I\subset R$ be an ideal.
Let $P_1,\ldots ,P_r$ be the minimal prime ideals of $R$ such that $\dim R/P_i=d$. Then
\begin{equation}\label{eqb1}
a\epsilon(I)=\sum_{i=1}^r\ell_{R_{P_i}}(R_{P_i})a\epsilon(I(R/P_i)).
\end{equation}
\end{Theorem}

\begin{proof} First suppose that $d=0$. In this case, $R$ is Artinian and local. Its only prime ideal is $m_R$. If $I=R$, then both sides of (\ref{eqb1}) are zero. If $I\ne R$ then $I$ is nilpotent, and every power of $I$ has saturation $R$ because $m_R^N=0$ for some $N$. Thus $a\epsilon(I)=\ell_R(R)$. $a\epsilon(0)=1$ so that $a\epsilon(I)=\ell_R(R)a\epsilon(0)$, showing that (\ref{eqb1}) holds. We may thus assume that $d>0$.

First assume that $R$ is complete. Let $Q$ be a minimal prime of $R$ with $\dim R/Q=d$. Set $A=R/Q$ and  $\overline I=IA$. Let $J_n= (I^n)^{\rm sat}$, $L_n=J_nA$ and $H_n=(\overline I^n)^{\rm sat}$. Thus
\begin{equation}\label{eqb2}
\overline I^n\subset L_n\subset H_n
\end{equation}by (\ref{eqa6}).
$H_n/L_n$ has finite length, since the localization $(H_n/L_n)_P=0$ if $P\in \mbox{Spec}(A)\setminus{m_A}$. We will show that
\begin{equation}\label{eqb3}
\lim_{n\rightarrow \infty}\frac{a(\overline I^n,L_n)}{n^d}=a\epsilon(\overline I).
\end{equation}
If $\overline I=0$, then $L_n=0$ for all $n$ by (\ref{eqb2}). Thus $a(\overline I^n,L_n)=0$ for all $n$ which implies that the limit (\ref{eqb3}) exists and is zero. We may thus assume that $\overline I\ne 0$.

By Lemma 5.2 \cite{CL}, there exists an integer $c>0$ such that 
\begin{equation}\label{eqb5}
\overline I^{mk}\cap m_A^{cmk}=H_m^k\cap m_A^{cmk}\mbox{ for }m,k\ge 1.
\end{equation}
Since 
$$
\overline I^{mk}= (\overline I^m)^k\subset L_m^k\subset H_m^k,
$$
Equation (\ref{eqb5}) implies that
\begin{equation}\label{eqb6}
L_m^k\cap m_A^{cmk}=H_m^k\cap m_A^{cmk}.
\end{equation}
It also shows that
\begin{equation}\label{eqb7}
(L_{mk}\cap H_m^k)\cap m_A^{cmk}=H_m^k\cap m_A^{cmk}.
\end{equation}
Since $R_Q$ is a zero dimensional local ring, it nilradical $QR_Q$ is nilpotent. 
Choose $N\ge 1$ such that $(QR_Q)^N=0$. As $Q^N$ is finitely generated, there exists $s\in R\setminus Q$ such that
\begin{equation}\label{eqb8}
sQ^N=0.
\end{equation}
Write $\overline s$ for the image of $s$ in $A$.

Choose $g\in I\setminus Q$. For each $m\ge 1$, let $b_m=\overline s\,\overline g^{m(N-1)}\in A$.

We now show that for $k$ sufficiently large, 
\begin{equation}\label{eqb10}
b_mH_m^k\subset L_{mk}.
\end{equation}
Let $\tilde H_m$ be the inverse image of $H_m$ in $R$. Since $H_m$ is the saturation of $\overline I^m$, there exist $q_m$ such that $m_R^{q_m}\tilde H_m\subset I^m+Q$. Hence for $k>0$,
$$
m_R^{kq_m}\tilde H_m^k\subset (I^m+Q)^k=\sum_{j=0}^kI^{m(k-j)}Q^j.
$$
Multiplying by $s$, all terms with $j\ge N$ are annihilated by (\ref{eqb8}). For $j<N$, 
$$
I^{m(k-j)}\subset I^{m(k-N+1)}.
$$
 Thus
 $$
 sm_R^{kq_m}\tilde H_m^k\subset I^{m(k-N+1)}\mbox{ for $k\ge N$.}
 $$
 Modding out by $Q$ gives $\overline sH_m^k\subset L_{m(k-N+1)}$. Since $\overline g^{m(N-1)}\in \overline I^{m(N-1)}\subset L_{m(N-1)}$, the fact that $\{L_n\}$ is a graded family gives that
 $$
 b_mH_m^k\subset L_{m(N-1)}L_{m(k-N+1)}\subset L_{mk},
 $$
 establishing equation (\ref{eqb10}).


 For each fixed $m$, define graded families in $A$ by
 \begin{equation}\label{eqb11}
 U(m)_k=L_{mk}\cap H_m^k\mbox{ and }V(m)_k=H_m^k.
 \end{equation}
 These families are graded since $U(m)_kU(m)_l\subset L_{m(k+l)}\cap H_m^{k+l}=U(m)_{k+l}$.
 They satisfy $U(m)_1=L_m$, $V(m)_1=H_m$ and $U(m)_k\subset L_{mk}$, $V(m)_k\subset H_{mk}$. By (\ref{eqb7}),
 \begin{equation}\label{eqb12}
 U(m)_k\cap m_A^{cmk}=V(m)_k\cap m_A^{cmk}\mbox{ for all }m,k.
 \end{equation}
 $\{L_n\}$ is a filtration, so the minimal prime condition of Theorem \ref{TheoremCL} is satisfied.
Thus these families satisfy the hypotheses of Theorem  \ref{TheoremCL}. So for fixed $m$, the limit 
$$
G(m)=\lim_{k\rightarrow \infty}\frac{\ell_A(V(m)_k/U(m)_k)}{k^d/d!}
$$
 exists and the limit
 $$
 F=\lim_{n\rightarrow \infty}\frac{\ell_R(H_n/L_n)}{n^d/d!}
 $$
 exists with
 \begin{equation}\label{eqb*}
 F=\lim_{m\rightarrow\infty}\frac{G(m)}{m^d}.
 \end{equation}
 We now show that for fixed $m$,
 \begin{equation}\label{eqb13}
 \lim_{k\rightarrow\infty}\frac{\ell_A(V(m)_k/U(m)_k)}{k^d}=0.
 \end{equation}
 By (\ref{eqb10}), for $k\gg 0$, $b_mH_m^k\subset U(m)_k$. By (\ref{eqb12}), $m_A^{cmk}H_m^k\subset U(m)_k$. Thus
 \begin{equation}\label{eqb14}
 \ell_A(V(m)_k/U(m)_k)\le \ell_A(H_m^k/b_mH_m^k+m_A^{cmk}H_m^k).
 \end{equation}
 
 Set $H=H_m$, $b=b_m$ and $a=cm$. We will bound $\ell_A(H^k/bH^k+m_A^{ak}H^k)$ for $k\gg 0$. let
 $\mathcal R=\oplus_{k\ge 0}H^k$
 and $\mathcal M=\mathcal R/bR=\oplus_{k\ge 0}M_k$ where $M_k=H^k/bH^k$. $\mathcal M$ is a finitely generated graded module over $\mathcal R$. We have
 \begin{equation}\label{eqc3}
 M_k/m_A^{ak}M_k\cong H^k/(bH^k+m_A^{ak}H^k).
 \end{equation}
 Let 
 \begin{equation}\label{eqc4}
 F(q,k)=\ell_A(M_k/m_A^qM_k).
 \end{equation}
 We will approximate the function $F(ak,k)$. Let $\kappa=A/m_A$. Set
 $\mathcal G=\oplus_{p,k\ge 0}\mathcal G_{p,q}$ where $\mathcal G_{p,q}=m_A^pM_k/m_A^{p+1}M_k$.
$\mathcal G$ is a finitely generated bi-graded module over the standard bigraded $\kappa$-algebra
 $$
 \mathcal B=\oplus_{p,k\ge 0}m_A^pH^k/m_A^{p+1}H^k.
 $$
  Let $X$ be a variable with $\deg X=(1,0)$. Set $\mathcal H=\mathcal G[X]$, which is a finitely graded module over the standard bigraded $\kappa$-algebra 
$\mathcal B[X]$. The component of bidegree $(q-1,k)$ of $\mathcal H$ is
$$
\mathcal H_{q-1,k}=\sum_{p=0}^{q-1}\mathcal G_{p,k}X^{q-1-p}.
$$
We have that
\begin{equation}\label{eqc8}
F(q,k)=\dim_{\kappa}\mathcal H_{q-1,k}.
\end{equation}

Let $D=\dim \mathcal G$. We have that $\mathcal B=\mbox{gr}_{m_A\mathcal R}(\mathcal R)$ and $\mathcal G=\mbox{gr}_{m_A\mathcal R}(\mathcal M)$. Thus 
$$
D=\dim{\rm gr}_{m_A\mathcal R}(\mathcal M)\le \dim \mathcal M
$$
by Theorem 4.5.6 \cite{BH}.
There is a natural graded inclusion $\mathcal R\subset A[Y]$ where $Y$ is a variable with $\deg Y=1$. Multiplication by $b$ is injective on $\mathcal R$, which implies that $\dim \mathcal M =\dim \mathcal R/b\mathcal R\le d$. Thus $D\le d$ and so
\begin{equation}\label{eqc12}
\dim\mathcal H\le D+1\le d+1.
\end{equation}
Since $\mathcal H$ is a finitely generated bigraded module over the standard bigraded $\kappa$-algebra $\mathcal B[X]$,
For $s,t\gg 0$, the function $(s,t)\mapsto \dim_{\kappa}(\mathcal H_{s,t})$ is a polynomial in $s$ and $t$ of total degree at most $r\dim \mathcal H-2$,
where 
$$
r\dim\mathcal H=\max\{\dim \mathcal B[X]/P\mid P\in \mbox{Proj}(\mathcal B[X])\mbox{ and }\mathcal H_P\ne 0\},
$$
as follows from \cite{HHRT} or Theorem 2.6 \cite{TV}. By (\ref{eqc8}), we have that for $q,k\gg 0$, $F(q,k)=P(q,k)$ for a polynomial $P(q,k)$ whose total degree in $q$ and $k$ is less than or equal to $r\dim \mathcal H-2$. 
Now $\mbox{Supp}_{\mathcal B[X]}\mathcal H=V(\mbox{Ann}_{\mathcal B[X]}\mathcal H)$ and $\dim\mathcal H=\dim_{\mathcal B[X]}\mathcal H=\dim(\mathcal B[X]/\mbox{Ann}_{\mathcal B[X]}\mathcal H)$. So $\mathcal H_P\ne 0$ implies $\mbox{Ann}_{\mathcal B[X]}\mathcal H\subset P$. The natural surjection $\mathcal B[X]/\mbox{Ann}_{\mathcal B[X]}\mathcal H\rightarrow \mathcal B[X]/P$ implies that $\dim\mathcal B[X]/P\le \dim_{\mathcal B[X]}\mathcal H$. Thus $r\dim \mathcal H\le \dim \mathcal H$, and so by (\ref{eqc12}), the total degree of $P(q,k)$ is less than or equal to $\dim\mathcal H-2\le d-1$.
Thus
$F(ak,k)=P(ak,k)$ is bounded by a polynomial in $k$ of degree $d-1$. So for fixed $m$,
$$
\lim_{k\rightarrow\infty}\frac{\ell_A(H^k_m/b_mH_m^k+m_A^{cmk}H_m^k)}{k^d}=\lim_{k\rightarrow\infty}\frac{\ell_A(M_k/m_A^{ak}M_k)}{k^d}=\lim_{k\rightarrow\infty}\frac{F(ak,k)}{k^d}=0.
$$
Therefore (\ref{eqb14}) implies that for fixed $m$,
$$
\lim_{k\rightarrow \infty}\frac{\ell_A(V(m)_k/U(m)_k)}{k^d}=0,
$$
and we have established  (\ref{eqb13}). By (\ref{eqb*}), we have by (\ref{eqb13}) that
\begin{equation}\label{eqb15}
\lim_{n\rightarrow\infty}\frac{\ell_A(H_n/L_n)}{n^d/d!}=0.
\end{equation}

Now for positive integers $m$, define graded families of ideals in $A$ by $U'(m)_k=L_m^k$ and $V'(m)_k=H_m^k$. These families satisfy $U'(m)_k\subset V'(m)_k$, $U'(m)_1=L_m$, $V'(m)_1=H_m$, $U'(m)_k\subset L_{mk}$ and $V'(m)_k\subset H_{mk}$. Equation (\ref{eqb6}) implies $L_m^k\cap m_A^{cmk}=H_m^k\cap m_A^{cmk}$. We have that $\{L_n\}$ is a filtration, so the minimal prime condition of Theorem \ref{TheoremCL} is satisfied.

Thus these families satisfy the hypotheses of Theorem \ref{TheoremCL}, and so
\begin{equation}\label{eqb16}
\lim_{n\rightarrow\infty}\frac{\ell_A(H_n/L_n)}{n^d/d!}=\lim_{m\rightarrow\infty}\frac{1}{m^d}\left(\lim_{k\rightarrow\infty}\frac{\ell_A(H_m^k/L_m^k)}{k^d/d!}\right)
=\lim_{m\rightarrow \infty}\frac{a(L_m,H_m)}{m^d}.
\end{equation}
Combining (\ref{eqb15}) and (\ref{eqb16}), we obtain
\begin{equation}\label{eqb17}
\lim_{m\rightarrow\infty}\frac{a(L_m,H_m)}{m^d}=0.
\end{equation}

For every $q$, we have short exact sequences
$$
0\rightarrow L_m^q/\overline I^{mq}\rightarrow H_m^q/\overline I^{mq}\rightarrow H_m^q/L_m^q\rightarrow 0.
$$
Thus 
\begin{equation}\label{eqb18}
a(\overline I^m,H_m)=a(\overline I^m,L_m)+a(L_m,H_m).
\end{equation}
Dividing by $m^d$, taking limits as $m$ goes to infinity and using (\ref{eqb17}), and the fact that
$$
\lim_{m\rightarrow\infty}\frac{a(\overline I^m,H_m)}{m^d}=a\epsilon(\overline I),
$$
 we obtain (\ref{eqb3}).
 
 By Theorem \ref{TheoremA2},
 \begin{equation}\label{eqb19}
 a(I^n,J_n)=\sum_{j=1}^r\ell_{R_{P_j}}(R_{P_j})a(I^n(R/P_j),J_n(R/P_j)). 
\end{equation}
Dividing by $n^d$ and taking the limit as $n$ goes to infinity, we get from (\ref{eqb3}) that 
\begin{equation}\label{eqb20}
a\epsilon(I)=\sum_{j=1}^r\ell_{R_{P_j}}(R_{P_j})a\epsilon(I(R/P_j)),
\end{equation}
giving the desired formula (\ref{eqb1}) of the statement of the theorem for every complete Noetherian local ring of dimension $d>0$.

Now let $R$ be an arbitrary Noetherian local ring of dimension $d>0$ and let $I$ be an ideal of $R$. Let $T=\hat R$. We have that
$(I^n)^{\rm sat}\hat R=((I\hat R)^n)^{\rm sat}$ by (\ref{eqa}) and 
$$
a(I^n,(I^n)^{\rm sat})=a((IT)^n,(IT)^n)^{\rm sat})
$$
by (\ref{eqa1}). Thus 
\begin{equation}\label{eqb22}
a\epsilon(I)=a\epsilon(IT)
\end{equation}
and for every quotient $R/P$ of $R$ by a minimal prime of $R$ with $\dim R/P=d$,
\begin{equation}\label{eqb23}
a\epsilon(I(R/P))=a\epsilon(I(T/PT)).
\end{equation}

For $P\in\mbox{Min}_d(R)$, let 
$$
\Lambda(P)=\{Q\in\mbox{Min}_d(T)\mid Q\cap R=P\}.
$$
For $Q\in \Lambda(P)$, the closed fiber $T_Q/PT_Q\cong (T/PT)_Q$ of the flat local homomorphism $R_P\rightarrow T_Q$   has dimension zero. Set
$$
\lambda_P=\ell_{R_P}(R_P)\mbox{ and }\mu_Q=\ell_{(T/PT)_Q}((T/PT)_Q).
$$
Tensoring a composition series of the Artinian local ring $R_P$ with the flat $R_P$-module $T_Q$ gives 
\begin{equation}\label{eqb24}
\ell_{T_Q}(T_Q)=\lambda_P\mu_Q.
\end{equation}
Applying (\ref{eqb20}), we get
$$
a\epsilon(IT)=\sum_{P\in\mbox{Min}_d(R)}\sum_{Q\in \Lambda(P)}\ell_{T_Q}(T_Q)a\epsilon(I(T/Q)).
$$
Using (\ref{eqb24}), we get
\begin{equation}\label{eqb25}
a\epsilon(I,T)=\sum_{P\in\mbox{Min}_d(R)}\lambda_P\sum_{Q\in \Lambda(P)}\mu_Qa\epsilon(I(T/Q)).
\end{equation}

Applying (\ref{eqb20}) to $T/PT\cong\widehat{R/P}$ gives
$$
\sum_{Q\in\Lambda(P)}\mu_Qa\epsilon(I(T/Q))=a\epsilon(I(T/PT)).
$$
By (\ref{eqb23}), the right hand side of this equation is $a\epsilon(I(R/P))$. Substituting this into (\ref{eqb25}) and then using (\ref{eqb22}), we obtain
$$
a\epsilon(I)=\sum_{P\in\mbox{Min}_d(R)}\ell_{R_P}(R_P)a\epsilon(I(R/P)),
$$
proving (\ref{eqb1}) for arbitarary Noetherian $R$.

\end{proof}



We have natural identifications $H^0_{m_R}(R/I^q)\cong (I^q)^{\rm sat}/I^q$, interpreting $(I^q)^{\rm sat}/I^q$ as a local cohomology module.

\begin{Lemma}\label{LemmaA0} Let $R$ be a $d$-dimensional Noetherian local ring and $I\subset R$ be an ideal. Suppose that $d>0$.

Let $E=\oplus_{q\ge 0}I^q$ be the Rees algebra of $I$. Suppose that $C=\oplus_{q\ge 0}C_q$ is a finitely generated graded $E$-module. Then
there exists a constant $c$ such that
\begin{equation}\label{eqd2}
\ell_R(H^0_{m_R}(C_q))\le cq^{d-1}\mbox{ for all $q$}.
\end{equation}
\end{Lemma}

\begin{proof} Let $T=H^0_{m_RE}(C)$, so that $T_q=H^0_{m_R}(C_q)$ for all $q$. $T$ is a graded submodule of the finite $E$-module $C$, so it is finitely generated over the Noetherian ring $E$. Thus there exists $s>0$ such that $m_R^sT=0$, and so $T$ is a finitely generated graded module over the standard graded algebra $E/m_R^sE$, whose degree 0 ring is Artinian. We have
$$
\dim T\le \dim(E/m_R^sE)=\dim(E/m_RE)\le d.
$$
The last inequality is by Proposition 5.1.6 \cite{HS}, since $\dim(E/m_RE)$ is the analytic spread of $I$.
Finally, by the theory of Hilbert polynomials of graded algebras over a Artinian local ring, there exists a constant $c$ such that $\ell_R(T_q)\le cq^{d-1}$ for all $q$.
\end{proof}

\begin{Theorem}\label{TheoremA3}
Let $R$ be a $d$-dimensional Noetherian local ring and $I\subset R$ be an ideal. Then
\begin{equation}\label{eqd10}
\epsilon(I)\le \sum_{P\in{\rm Min}_d(R)}\ell_{R_P}(R_P)\epsilon(I(R/P).
\end{equation}
\end{Theorem}

\begin{proof} If $d=0$ and $I$ is a proper ideal of $R$, then $R$ is Artinian, $I$ is nilpotent and thus both sides of (\ref{eqd10}) equal $\ell_R(R)$. When $I=R$, both sides of (\ref{eqd10}) are zero. We may thus assume that $d>0$.

 For a finite $R$-module $M$, define 
$$
h_M(q)=\ell_R(H^0_{m_R}(M/I^qM),
$$
 and define
\begin{equation}\label{eqd3}
\epsilon(I;M)=\limsup_{q\rightarrow\infty}\frac{h_M(q)}{q^d/d!}.
\end{equation}
In particular, $\epsilon(I;R)=\epsilon(I)$.

Let 
$$
0\rightarrow L\rightarrow M\rightarrow N\rightarrow 0
$$
be a short exact sequence of finitely generated $R$-modules. For every $q$, there is an exact sequence 
$$
0\rightarrow L/L\cap I^qM\rightarrow M/I^qM\rightarrow N/I^qN\rightarrow 0.
$$
Thus
\begin{equation}\label{eqd5}
h_M(q)\le \ell_R(H^0_{m_R}(L/L\cap I^qM))+h_N(q).
\end{equation}
By Artin-Rees, there exists $c\ge 0$ such that setting $D=L\cap I^cM$, one has that 
\begin{equation}\label{eqd6}
L\cap I^qM=I^{q-c}D\mbox{ for }q\ge c.
\end{equation}
From (\ref{eqd6}), we obtain the short exact sequence
$$
0\rightarrow I^{q-c}L/I^{q-c}D\rightarrow L/L\cap I^qM\rightarrow L/I^{q-c}L\rightarrow 0.
$$
Hence
\begin{equation}\label{eqd7}
\ell_R(H^0_{m_R}(L/L\cap I^qM))\le h_L(q-c)+\ell_R((H^0_{m_R}(I^{q-c}L/I^{q-c}D)).
\end{equation}
Let $C=\oplus_{r\ge 0}I^rL/I^rD$. $C$ is a finitely generated graded $E=\oplus_{n\ge 0}I^n$-module. By (\ref{eqd2}), there exists a constant $c_1$ such that 
$$
\ell_R(H^0_{m_R}(I^{q-c}L/I^{q-c}D))\le c_1q^{d-1} 
$$
for all $q$. Now using (\ref{eqd5}) and (\ref{eqd7}), we obtain that
$$
h_M(q)\le h_L(q-c)+h_N(q)+c_1q^{d-1}
$$
for all $q$. Now divide this equation by $q^d/d!$ and take the limsup to see that 
 \begin{equation}\label{eqd4}
\epsilon(I;M)\le \epsilon(I;L)+\epsilon(I;N).
\end{equation}

Choose a finite prime filtration of $R$,
$$
0=M_0\subset M_1\subset \cdots\subset M_s=R,
$$
where $M_j/M_{j-1}\cong R/P_j$ for some prime ideals $P_j$ of $R$. Repeated application of (\ref{eqd4}) gives
\begin{equation}\label{eqd9}
\epsilon(I)\le \sum_{j=1}^s\epsilon(I;R/P_j).
\end{equation}
We have  that $\dim R/P_j< d$ implies $\epsilon(I;R/P_j)=0$ since $\epsilon(I(R/P_j))<\infty$. Since for $P$ a minimal prime of $R$, $R/P$ must occur $\ell_{R_P}(R_P)$ times as a quotient $M_j/M_{j-1}\cong R/P$, we have that
$$
\epsilon(I)\le \sum_{P\in {\rm Min}_d(R)}\ell_{R_P}(R_P)\epsilon(I(R/P)),
$$
establishing the theorem.

\end{proof}

We now prove Theorem \ref{TheoremA4}, showing that $\epsilon(I)=a\epsilon(I)$ for an ideal $I$ in a Noetherian local ring.

First suppose that  $d=0$. If $I$ is proper, both $\epsilon(I)$ and $a\epsilon(I)$ equal $\ell_R(R)$. If $I=R$ both
$\epsilon(I)$ and $a\epsilon(I)$ 
 are zero. We may thus suppose that $d>0$.

We have that
\begin{equation}\label{eqe1}
a\epsilon(I)\le \epsilon(I).
\end{equation}
 The proof is as follows. 
Let $H_n=(I^n)^{\rm sat}$. We have that
\begin{equation}\label{eqa11}
H_n^q\subset H_{nq}.
\end{equation}
For every $n\ge 1$,
$$
\frac{a(I^n,H_n)}{n^d}=\lim_{q\rightarrow\infty}\frac{\ell_R(H_n^q/I^{nq})}{(nq)^d/d!}\le \limsup_{q\rightarrow\infty}\frac{\ell_R(H_{nq}/I^{nq})}{(nq)^d/d!}
\le \limsup_{r\rightarrow\infty}\frac{\ell_R(H_r/I^r)}{r^d/d!}=\epsilon(I).
$$
Thus 
$$
a\epsilon(I)=\limsup_{n\rightarrow\infty}\frac{a(I^n/H_n)}{n^d}\le \epsilon(I).
$$
We also need the conclusions of Theorem \ref{TheoremA3},
\begin{equation}\label{eqe2}
\epsilon(I)\le \sum_{P\in{\rm Min}_d(R)}\ell_{R_P}(R_P)\epsilon(I(R/P)),
\end{equation}
as well as the formula of Theorem \ref{TheoremB2},
\begin{equation}\label{eqe3}
a\epsilon(I)=\sum_{P\in{\rm Min}_d(R)}\ell_{R_P}(R_P)a\epsilon(I(R/P)).
\end{equation}

We first prove the theorem for complete local rings. Suppose that $R$ is complete. Write $\mbox{Min}_d(R)=\{P_1,\ldots, P_r\}$. Set $A_i=R/P_i$, $I_i=IA_i$, 
and $\lambda_i=\ell_{R_{P_i}}(R_{P_i})$. Since each $A_i$ is analytically unramified, 
\begin{equation}\label{eqe4}
\epsilon(I_i)=a\epsilon(I_i)
\end{equation}
 by Theorem \ref{TheoremA}. Combining (\ref{eqe1}), (\ref{eqe2}), (\ref{eqe4}) and (\ref{eqe3}), we have 
 $$
 a\epsilon(I)\le \epsilon(I)\le\sum_{i=1}^r\ell_{R_{P_i}}(R_{P_i})\epsilon(I_i)=\sum_{i=1}^r\ell_{R_{P_i}}(R_{P_i})a\epsilon(I_i)=a\epsilon(I),
 $$
 so that 
 \begin{equation}\label{eqe5}
 \epsilon(I)=a\epsilon(I),
 \end{equation}
 establishing the theorem for complete $R$.
 
 We now prove the equality for an arbitrary Noetherian local ring. Let $R$ be arbitrary and put $T=\hat R$, $B=IT$. 
 We have that
 \begin{equation}\label{eqe7}
 \epsilon(I)=\epsilon(B).
 \end{equation}
 For all $n$, $(I^n)^{\rm sat}T=(B^n)^{\rm sat}$ by (\ref{eqa}). Further,
 $$
 a(I^n,(I^n)^{\rm sat})=a(B^n,(B^n)^{\rm sat})
 $$
for all $n$ by (\ref{eqa1}). Thus
\begin{equation}\label{eqe8}
a\epsilon(I)=a\epsilon(B).
\end{equation}
We have that 
$$
a\epsilon(I)=a\epsilon(B)=\epsilon(B)=\epsilon(I)
$$ by (\ref{eqe8}), (\ref{eqe5}) and (\ref{eqe7}), establishing Theorem \ref{TheoremA4} in complete generality.

\vskip .4truein
We now prove Theorem \ref{TheoremA5}, showing that $\epsilon(I)$ exists as a limit.

\vskip .2truein
Set $H_n=(I^n)^{\rm sat}$, $h(n)=\ell_R(H_n/I^n)=\ell_R(H^0_{m_R}(R/I^n))$. For each $r\ge 1$, put 
$$
A_r=\frac{a(I^r,H_r)}{r^d}.
$$
 We will prove that
\begin{equation}\label{eqf2}
\mbox{Sup}_rA_r\le \liminf_{n\rightarrow\infty}\frac{h(n)}{n^d/d!}
\end{equation}
which gives a stronger inequality than  (\ref{eqe1}). Once (\ref{eqf2}) is established, we have that

\begin{equation}\label{eqf3}
a\epsilon(I)=\limsup_{r\rightarrow\infty}A_r\le \mbox{Sup}_{r}A_r\le \liminf_{n\rightarrow\infty}\frac{h(n)}{n^d/d!}
\le \limsup_{n\rightarrow\infty}\frac{h(n)}{n^d/d!}=\epsilon(I).
\end{equation}
Since $a\epsilon(I)=\epsilon(I)$ by Theorem \ref{TheoremA4}, we have established Theorem \ref{TheoremA5}.

It remains to prove (\ref{eqf2}). Define $b(n)=\ell_R(H^0_{m_R}(I^n/I^{n+1}))$. We deduce from the exact sequence
$$
0\rightarrow I^n/I^{n+1}\rightarrow R/I^{n+1}\rightarrow R/I^n\rightarrow 0
$$
and left exactness of $H^0_{m_R}$ that
\begin{equation}\label{eqf4}
h(n+1)\le h(n)+b(n).
\end{equation} 
We have that there exists a constant $c$ such that 
\begin{equation}\label{eqf5}
b(n)\le cn^{d-1}
\end{equation}
for all $n$ by Lemma \ref{LemmaA0}. Combining (\ref{eqf4}) and (\ref{eqf5}), we see that 
\begin{equation}\label{eqf6}
h(n+1)\le h(n)+cn^{d-1}.
\end{equation}

Now fix $r$.
For every $q\ge 1$, $I^{rq}\subset H_r^q\subset H_{rq}$  and so
\begin{equation}\label{eqf8}
h(rq)\ge \ell_R(H_r^q/I^{rq}).
\end{equation}
Now 
\begin{equation}\label{eqf9}
a(I^r,H_r)=\lim_{q\rightarrow\infty}\frac{\ell_R(H_r^q/I^{rq})}{q^d/d!}.
\end{equation}
Dividing (\ref{eqf8}) by $(rq)^d/d!$ gives 
$$
\liminf_{q\rightarrow\infty}\frac{h(rq)}{(rq)^d/d!}\ge \frac{a(I^r,H_r)}{r^d}=A_r.
$$
For an arbitrary large $n$, Let $N=r\lceil \frac{n}{r}\rceil$. $N$ is the smallest multiple of $r$ satisfying $N\ge n$. We have that   $0\le N-n<r$.

Iterating (\ref{eqf6}) from $n$ to $N$ gives
$$
h(N)\le h(n)+c(\sum_{j=n}^{N-1}j^{d-1}).
$$
Since $r$ is fixed, $h(N)\le h(n)+c_1n^{d-1}$ for some constant $c_1$, so 
\begin{equation}\label{eqf11}
h(n)\ge h(N)-c_1n^{d-1}.
\end{equation}
Let $q=\lceil\frac{n}{r}\rceil$. Equations (\ref{eqf8}) and (\ref{eqf11}) imply 
\begin{equation}\label{eqf12}
h(n)\ge \ell_R(H_r^q/I^{rq})-c_1n^{d-1}.
\end{equation}
Dividing by $n^d/d!$, we have that
\begin{equation}\label{eqf13}
\frac{h(n)}{n^d/d!}\ge \frac{\ell_R(H_r^q/I^{rq})}{q^d/d!}\left(\frac{q}{n}\right)^d-\frac{c_1}{n}.
\end{equation}
Taking the limit in (\ref{eqf13}) as $n\mapsto \infty$ we have that $q=\lceil \frac{n}{r}\rceil\mapsto \infty$ and $\frac{q}{n}\mapsto \frac{1}{r}$. Using (\ref{eqf9}) in (\ref{eqf13}), we obtain
\begin{equation}\label{eqf14}
\liminf_{n\rightarrow\infty}\frac{h(n)}{n^d/d!}\ge \frac{a(I^r,H_r)}{r^d}=A_r.
\end{equation}
Now (\ref{eqf14}) holds for every $r\ge 1$. Taking the supremum proves (\ref{eqf2}).

\section{Amao multiplicities and Segre products}\label{AmaoSegre}

\subsection{Intersection Theory} We use the intersection theory defined by Thorup in \cite{Th}. This intersection theory depends on the choice of a ``grading'' on a Noetherian base scheme $X$, which gives  an ``induced grading'' on schemes of finite type over $X$. On a general Noetherian scheme $X$, 
the ``canonical grading'' is defined by $\delta(V)=-\dim\mathcal O_{X,V}$ if $V\subset X$ is a closed integral subscheme of $X$, where $\mathcal O_{X,V}$ is defined to be the stalk $\mathcal O_{X,\eta}$ at the generic point $\eta$ of $V$.
If $X$ has finite Krull dimension, then $X$ has the ``topological grading'', which is defined by $\delta(V)=\dim V$, if $V\subset X$ is a closed integral subscheme of $X$. We will assume from now on that $X$ has the topological grading. In this manuscript, we will be concerned with $X=\mbox{Spec}(R)$ where $R$ is a Noetherian local ring, which has finite Krull dimension.

If $U\subset X$ is a subset, then $\delta(U)$ is defined to be $\delta(U)=\sup_{u\in U}\delta(\overline{\{u\}})$. Suppose that $f:Y\rightarrow X$ is essentially of finite type. 
Then the induced grading on $Y$ is defined for $V\subset Y$ a closed  integral subscheme by
$$
\delta(V)=\mbox{trdeg}_{R(W)}R(V)+\dim W
$$
where $W$ is the closure of $f(V)$ in $X$ and $R(V)$, $R(W)$ are the respective function fields of $V$ and $W$. Thus $\dim V\le \delta(V)$ by the dimension inequality (Theorem 23, page 96 \cite{Ma}), with equality if $X$ is universally catenary.

If $f:Y\rightarrow X$ is proper, then $\delta(V)=\dim V$ for all integral subschemes $V$ of $Y$. This is proven in Proposition 4.3 \cite{Th}.

In this manuscript, we will exclusively be concerned with proper morphisms $Y\rightarrow X=\mbox{Spec}(R)$, where $R$ is a Noetherian local ring. Thus we may define $Z_r(Y)$ to be the group of cycles of dimension $r$ on $Y$, which map to $A_r(Y)$, which is the quotient of $Z_r(Y)$ by the rational equivalence defined in \cite{Th}. The theory of rational equivalence defined in \cite{Th} determines an intersection theory on $Y$, which we will use.
 
 If $S$ is a cycle on $Y$, we define $[S]_r$ to be the $r$-dimensional part of $S$.

 If $f:Z\rightarrow Y$ is a morphism of finite type $X$-schemes, and $V$ is an integral subscheme of $Z$ and $W=f(V)$, then
 $$
 f_*(V)=
 \left\{
 \begin{array}{ll} 
 [R(W):R(V)]W&\mbox{ if }\dim W=\dim V\\
 0&\mbox{ otherwise}
 \end{array}\right.
 $$
 where $R(V)$ and $R(W)$ are the respective function fields of $V$ and $W$. This induces a map on cycles.
 
In Proposition 6.5 \cite{Th}, the following basic fact is proven.  Let $Z$ and $Y$ be proper $X$-schemes and $f : Z\rightarrow  Y$
be a proper $X$-morphism. Then the homomorphism $f_* : Z_k(Z)\rightarrow Z_k(Y)$ defined above induces a
functorial homomorphism $f_* : A_k(Z/X)\rightarrow A_k(Y/X)$ for all $k$. 

Let $(R,m_R)$ be a $d$-dimensional Noetherian local ring and $X=\mbox{Spec}(R)$. Suppose that $Y\rightarrow X$ and $W\rightarrow X$ are  projective $R$-schemes which are birationally equivalent to $X$ such that there exists an $R$-morphism $\beta:Y\rightarrow W$. If $U$ is a Cartier divisor on $W$, we will identify $\beta^*(U)$ with $U$. We are able to do this since intersections products are preserved by pull backs of Cartier divisors. 

If $F_1,\ldots,F_d$ are Cartier divisors on $Y$, we will denote the intersection product
$$
\int F_1\cdot\ldots\cdot F_d\cdot Y \mbox{ by }
(F_1\cdot\ldots\cdot F_d).
$$

\subsection{Amao multiplicities}
Suppose that $(R,m_R)$ is a Noetherian local ring of dimension $d$ and $K\subset J$ are ideals of $R$ such that the length $\ell_R(J/K)<\infty$.

We first introduce notation and review some material from the discussion on page 172 of \cite{KT}. Let $G'\subset G$ be the graded rings $G'=\sum_{n\ge 0}K^n$ and $G=\oplus_{n\ge 0}J^n$. Let $X=\mbox{spec}(R)$, $P=\mbox{Proj}(G)\stackrel{\pi}{\rightarrow} X$ be the blowup of $J$ and $P'=\mbox{Proj}(G')\rightarrow X$ be the blowup of $K$. We have that
$G_1'G=KG(-1)$ and 
\begin{equation}
\label{eq9}
\mathcal O_P(1)=J\mathcal O_P=\mathcal  O_P(-E)
\end{equation}
 for some effective Cartier divisor $E$ on $P$. Thus the sheafification of $G_1'G$ is
$$
\widetilde{G_1'G}\cong \widetilde{KG(-1)}\cong (\widetilde{KG})\otimes\mathcal O_P(-1)\cong (K\mathcal O_P)\otimes \mathcal O_P(E).
$$
Now $K\mathcal O_P\subset J\mathcal O_P=\mathcal O_P(-E)$. Thus $\widetilde{G_1'G}=\mathcal I_Z$ for the closed subscheme $Z$ of $P$ with ideal sheaf $\mathcal I_Z=(K\mathcal O_P)\otimes \mathcal O_P(E)$. Since $K_q=J_q$ for $q\in X\setminus \{m_R\}$, we have that $\mathcal I_{Z,y}=\mathcal O_{P,y}$ for $y\in X\setminus \pi^{-1}(m_R)$. Thus $Z$ is supported in $\pi^{-1}(m_R)$. 

Let $b:B\rightarrow P$ be the blowup of $Z$. Then $\mathcal I_Z\mathcal O_B=\mathcal O_B(-D)$ for some effective Cartier divisor $D$ on $B$, whose support contracts to $m_A$. Then
\begin{equation}\label{eq8}
K\mathcal O_B=(\mathcal I_Z\mathcal O_B)\otimes b^*\mathcal O_P(-E)\cong \mathcal O_B(-D-\pi^*E)
\end{equation}
is an invertible sheaf, so there is a natural morphism of $X$-schemes $\alpha:B\rightarrow P'$. 

Let $S=[\tilde G]_d=[\mathcal O_P]_d\in A_d(P)$. Write $S=\sum_{j=1}^ln_j[S_j]$ where $S_j$ are the integral components of $P$  of maximal dimension $d$. Now $\dim \pi^{-1}(m_A)\le d-1$ so no components of $S$ are contained in the support of $Z$, and so $S^Z=0$, where $S^Z$ is the cycle of components of $S$ which are supported in $Z$. Further, $b^+S=\sum_{j=1}^ln_j[b^+S_j]$ where $b^+S_j$ is the closure of $b^{-1}(S_j\setminus Z)$ in $B$. Thus $b^+S=B$. 

The following theorem is a consequence of Theorem 5.7 \cite{KT}, giving an interpretation of the Amao multiplicity $a(K,J)$ in terms of intersection products on $B$.

\begin{Theorem}\label{Theorem1}(Theorem 5.7 \cite{KT})
The Amao multiplicity $e=a(K,J)$, which is the limit
$$
\lim_{n\rightarrow \infty}\frac{\ell_R(J^n/K^n)}{n^d}=d!e,
$$
satisfies  $e=\sum_{i=1}^ds_i$, with $s^i=\int (\ell')^{i-1}\cdot \ell^{d-i}\cdot D\cdot B$, where $\ell'=\alpha^*\mathcal O_{P'}(1)\cong \mathcal O_B(-D-\pi^*E)$ and $\ell=b^*\mathcal O_P(1)\cong \mathcal O_B(-(b\pi)^*E)$.
\end{Theorem}

\begin{proof} The theorem  follows from Theorem 5.7 \cite{KT}, since  we can take the $k_0$ in that theorem to be 0 since $(K^n)_q=(J^n)_q$ for all $n>0$ if $q\in X\setminus \{m_R\}$.
\end{proof}

\begin{Remark}\label{Rem1} With the notation of Theorem \ref{Theorem1}, let $P_1,\ldots, P_r$ be the minimal primes of $R$ such that $\dim R/P_i=d$. Then
the associativity formula of Theorem \ref{TheoremA2} follows easily in this context; that is, 
\begin{equation}\label{eq24}
\lim_{n\rightarrow \infty}\frac{\ell_R(J^n/K^n)}{n^d}=\sum_{i=1}^r\ell_{R_{P_i}}(R_{P_i})\lim_{n\rightarrow\infty} \frac{\ell_{R/P_i}(J^n(R/P_i)/K^n(R/P_i))}{n^d}.
\end{equation}
\end{Remark}

\begin{proof} Let $B_i$ be the strict transform of $\mbox{Spec}(R/P_i)$ in $B$ for $1\le i\le r$. As a cycle,
$[X]_d=\sum \ell_{R_{P_i}}(R_{P_i})[R/P_i]_d$, so $[B]_d=\sum \ell_{R_{P_i}}(R_{P_i})[B_i]_d$. Now the Segre products 
$$
s^i=\int (\ell')^{i-1}\cdot \ell^{d-i}\cdot D\cdot [B]_d=\sum \ell_{R_{P_i}}(R_{P_i})\int (\ell')^{i-1}\cdot \ell^{d-i}\cdot D\cdot [B_i]_d
$$
and formula (\ref{eq24}) follows.
\end{proof}

\subsection{Comparison of Amao multiplicities} \label{SecComp}

Suppose that $I$ is an ideal of $R$ such that $I_p\ne 0$ if $p$ is an ideal of $R$ such that $I_p\ne 0$ if $p$ is an ideal of $R$ such that $\dim R/p=d$.

We will use similar notation when discussing birational morphisms of projective $k$-varieties.

 For $n>0$, let $K$  be $I^n$ and $J$  be $(I^n)^{\rm sat}$, in the notation of  Section \ref{AmaoSegre}. Let   $F$ be the Cartier divisor 
 $I\mathcal O_{P'}=\mathcal O_{P'}(-F)$, where  $P'$ is the blow up of $I$, so that  $\mathcal O_{P'}(-nF)=I^n\mathcal O_{P'}$.
  Then write $P_n$ for $P$,  $B_n$ for $B$,  $E_n$ for  $E$ and $D_n$ for $D$.  
 Thus   $\mathcal O_{B_n}(-E_n)=(I^n)^{\rm sat}\mathcal O_{B_n}$ and 
  \begin{equation}\label{eq17}
  -nF=-D_n-E_n.
  \end{equation}
 
 Now for $m,n> 0$, $(I^m)^{\rm sat}(I^n)^{\rm sat}\subset (I^{m+n})^{\rm sat}$, so making the comparison on a suitable projective $X$-scheme which birationally dominates both $B_m$ and $B_n$, we have that 
 \begin{equation}\label{eq3}
 \begin{array}{l}
 -E_m-E_n\le -E_{m+n},\\
 D_m+D_n=mF-E_m+nF-E_n\le (m+n)F-E_{m+n}=D_{m+n}.
 \end{array}
 \end{equation}
 
 We have commutative diagrams of $X$-morphisms

 \begin{equation}\label{eq4}
 \begin{array}{rccl}
 B_n&\stackrel{b_n}{\rightarrow}&P_n\\
 \alpha_n\downarrow&&\downarrow \pi_n\\
 P'&\stackrel{\beta}{\rightarrow}&X
 \end{array}
 \end{equation}

 We now suppose that $R$ is essentially of finite type over a field $L$. Let $\overline t_1,\ldots,\overline t_r$ be a transcendence basis of $R/m_R$ over $L$, and let $t_1,\ldots,t_r$ be lifts of the $\overline t_i$ to $R$. Then the rational function field $k=L(t_1,\ldots,t_r)$ is contained in $R$ and $R/m_R$ is finite over $k$.
Further, there exists a projective $k$-variety $\overline X$ and a closed point $y_0\in \overline X$ such that $\mathcal O_{\overline X,y_0}=R$. 
 Extend $I$ to an ideal sheaf $\mathcal I$ on $\overline X$ so that $\mathcal I_{y_0}=I$. Then for all $n>0$, $J_n$ extends naturally to an ideal sheaf $\mathcal J_n$ on $\overline X$ such that $(\mathcal J_n)_{y_0}=J_n$ and $(\mathcal J_n)_y=\mathcal I_y$ if $y\in \overline X\setminus \{y_0\}$.
 
The diagrams (\ref{eq4}) extends naturally to  commutative diagrams of birational projective $k$-morphisms
 
 \begin{equation}\label{eq5}
 \begin{array}{rccl}
 \overline B_n&\stackrel{\overline b_n}{\rightarrow}&\overline P_n\\
 \overline\alpha_n\downarrow&&\downarrow\overline\pi_n\\
 \overline P'&\stackrel{\overline{\beta}}{\rightarrow}&\overline X
 \end{array}
 \end{equation}

where $\overline P_n\rightarrow\overline X$ is the blowup of $\mathcal J_n$, $\overline{P'}\rightarrow \overline X$ is the blow up of $\mathcal I$,
 and $E_n$ and $F$ extend naturally to effective Cartier divisors
$\overline E_n$ and $\overline{F}$ on $\overline P_n$ and $\overline P'$ respectively.  $D_n=nF-E_n$ can be regarded as an effective   Cartier divisor on $\overline{B_n}$, which satisfies
$D_n=n\overline F-\overline E_n$.
$Z_n$ can be regarded as a scheme on $\overline P_n$ since it's support contracts to $m_R$. $\overline B_n\rightarrow \overline P_n$ is then the blow up of $\overline Z_n$.

There exists an effective ample Cartier divisor $H$ on $\overline X$ such that $y_0$ is not in the support of $H$ such that $y_0$ is not in the support of $H$ and $-\overline A:= \overline\beta^*(H)-\overline F$ is ample on $\overline{P'}$. Thus
$-\overline A$ is nef on any $R$-scheme which birationally dominates $\overline{P'}$. Then
\begin{equation}\label{eq6}
-n\overline A\cdot D_n=-nF\cdot D_n
\end{equation}
 for all $n$ since $D_n$ contracts to $m_R$. We have that
$$
nH-\overline{E_n}=-n\overline A+D_n
$$
and
\begin{equation}\label{eq7}
(-n\overline A+D_n)\cdot D_n=(-nF+D_n)\cdot D_n=-E_n\cdot D_n.
\end{equation}
\begin{Lemma}\label{Lemma1} $-n\overline A+D_n$ is nef for all $n>0$.
\end{Lemma}

\begin{proof} Suppose that $C$ is a closed curve on $\overline B_n$ which does not contract to $m_R$. Then $C$ is not contained in the support of the effective divisor $D_n$, so $(C\cdot D_n)\ge 0$.  Thus $(C\cdot (-n\overline A+D_n))\ge 0$, since $-n\overline A$ is nef. Suppose that $C$ is a closed curve on $\overline {B_n}$ which  contracts to $m_R$. Then $(C\cdot H)=0$, so $(C\cdot(-n\overline A+D_n))=(C\cdot(-E_n))\ge 0$.
\end{proof}

We will make extensive use of the following lemmas.
The proof of the following lemma is as in Example 1.4.16 \cite{Laz1}.

\begin{Lemma}\label{NefLem} Suppose that $Y\rightarrow \overline X$ is birational and projective and that 
$V$ is a closed  $r$ dimensional closed integral subscheme  of $Y$ which contracts to $m_R$ and $G_1,\ldots, G_{d-r}$ are nef Cartier divisors on $Y$. Then 
$(G_1\cdot\ldots\cdot G_{d-r}\cdot V)\ge 0$.
\end{Lemma}

\begin{Lemma}\label{Lemma4}
Suppose that $Y\rightarrow \overline X$ is birational and projective and that 
$F_i,G_i$ for $1\le i\le d$ are nef Cartier divisors on $Y$ such that  $F_i\ge G_i$ for $1\le i\le d$. Then
$$
(F_1\cdot\ldots\cdot F_d)\ge (G_1\cdot\ldots\cdot G_d).
$$
\end{Lemma}

\begin{proof} We have that 
$$
\left(F_1\cdot\ldots\cdot F_{i-1}\cdot(F_i-G_i)\cdot G_{i+1}\cdot\ldots\cdot G_d\right)\ge 0
$$
for $1\le i\le d$ by Lemma \ref{NefLem}.
since $F_i-G_i\ge 0$ and the $F_j$ and $G_j$ are nef.
The lemma now follows from  multilinearity of the intersection product.
\end{proof}

\subsection{Limits of nef divisors} Let notation be as in Section \ref{SecComp}.
\begin{Theorem}\label{Theorem2}
Let $a,b$ be nonnegative integers such that $a+b=d$. Then
$$
\lim_{n\rightarrow\infty}\frac{((-n\overline A+D_n)^a\cdot(-n\overline A)^b)}{n^d}\in \RR
$$
exists.
\end{Theorem}

\begin{proof}
$D_n=nF-E_n$ implies $D_n\le nF$ on $B_n$ which implies that $D_n\le n\overline F$ on $\overline B_n$. As explained after (\ref{eq5}), $I\mathcal O_{\overline P'}=\mathcal O_{\overline P'}(-\overline F)$. Let $H'$ be an ample effective divisor on $\overline P'$ such that $\overline F\le H'$. Thus $D_n\le nH'$ for all $n$. $-\overline A\le H$, so 
\begin{equation}\label{eq14}
-n\overline A+D_n\le n(H+H')
\end{equation}
 for all $n$. Thus since $-\overline A$,  $-n\overline A+D_n$, $H$ and $H'$ are nef, for nonnegative integers $a$ and $b$ such that $a+b=d$,
$$
((-n\overline A+D_n)^a\cdot(-n\overline A)^b)\le ((n(H+H'))^a\cdot (-nH)^b)=
n^d((H+H')^a\cdot(H^b))
$$
for all positive $n$, and so,
\begin{equation}\label{eq11}
\frac{((-n\overline A+D_n)^a\cdot(-n\overline A)^b)}{n^d}\le ((H+H')^a\cdot(H^b))
\end{equation}
for all positive $n$. 

 Let $U_n=-n\overline A+D_n$ for $n$ a positive integer
We have that
\begin{equation}\label{eq12}
U_m+U_n=(-m\overline A+D_m)+(-n\overline A+D_n)\le -(m+n)\overline A+D_{m+n}=U_{m+n}
\end{equation} 
for all positive integers $m$ and $n$
by (\ref{eq3}). 

Let
$$
S=\{\frac{(U_n^a\cdot(-n\overline A)^b)}{n^d}\mid n\in \ZZ_{>0}\}.
$$
$S\subset \RR_{\ge 0}$ by Lemma \ref{NefLem} since $-\overline A$ and $U_n$ are nef. Let $c$ be the least upper bound of $S$, which is a finite nonnegative real number by (\ref{eq11}).

Given $\epsilon>0$, there exists $m_1\in \ZZ_{>0}$ such that 
$$
\frac{(U_{m_1}^a\cdot(-m_1\overline A)^b)}{m_1^d}>c-\frac{\epsilon}{2}.
$$
$\frac{U_{nm_1}}{nm_1}\ge \frac{U_{m_1}}{m_1}$ for all $n\in \ZZ_{>0}$ by (\ref{eq12}). Since that $U_i$ and $-\overline A$ are nef,
\begin{equation}\label{eq13}
\frac{(U_{nm_1}^a\cdot(-nm_1\overline A)^b)}{(nm_1)^d}\ge \frac{(U_{m_1}^a\cdot(-m_1\overline A)^b)}{m_1^d}>c-\frac{\epsilon}{2}.
\end{equation}
For $m,i\in \ZZ_{>0}$, we have that 
\begin{equation}\label{eq22}
\begin{array}{lll}
\frac{U_i}{m+i}+\frac{U_m}{m+i}&=&\frac{1}{m+i}U_i+\frac{m}{m+i}(\frac{U_m}{m})\\
&=& \frac{1}{m+i}[U_i-i(\frac{U_m}{m})]+\frac{U_m}{m}
\end{array}
\end{equation}
which is  nef. Let
\begin{equation}\label{eq15}
W_{m_1n,i}:= \left((- \overline A)^b\cdot \left(\frac{U_{m_1n}+U_i}{m_1n+i}\right)^a\right).
\end{equation}
We have that
\begin{equation}\label{eq23}
\begin{array}{lll}
((-\overline A)^b\cdot \left(\frac{U_{m_1n+i}}{m_1n+i}\right)^a)
&\ge&((-\overline A)^b\cdot \left(\frac{U_{m_1n}+U_i}{m_1n+i}\right)^a)\\
&=& \left((-\overline A)^b\cdot \left[\frac{1}{m_1n+i}(U_i-i(\frac{U_{m_1n}}{m_1n}))+\frac{U_{m_1n}}{m_1n}\right]^a\right)
\end{array}
\end{equation}
where the first line is by (\ref{eq12}) and the second line is by (\ref{eq22}).

There exists $Q(x,y,z)\in \ZZ[x,y,z]$ such that setting $\overline x=U_i$, $\overline y=\frac{U_{m_1n}}{m_1n}$ and $\overline z=\frac{1}{m_1n+i}$,
$$
Q(\overline x,\overline y,\overline z)=W_{m_1n,i}-((-\overline A)^b\cdot(\frac{U_{m_1n}}{m_1n})^a).
$$
Expand $Q(x,y,z)=\sum a_{\alpha\beta\gamma}x^{\alpha}y^{\beta}z^{\gamma}$.
We have that $\alpha+\beta=d$ and $\gamma>0$ if $a_{\alpha\beta\gamma}\ne 0$. We have that
$$
0\le((U_i)^{\alpha}\cdot(\frac{U_{m_1n}}{m_1n})^{\beta})\le i^{\alpha}(H+H')^d
$$
by equation (\ref{eq14}). Let $\lambda=(m_1-1)^{\alpha}(H+H')^d$. Thus if $a_{\alpha\beta\gamma}\ne 0$, then 
$$
|\overline x^{\alpha}\overline y^{\beta}\overline z^{\gamma}|\le \lambda\left|\left(\frac{1}{m_1n}\right)\right|
$$
 for $n\in \ZZ_{\ge 0}$ and $0< i<m_1$. Thus given $\delta>0$, for $n\gg 0$,
 \begin{equation}\label{eq19}
 \left|\left((-\overline A)^b\cdot \left(\frac{U_{m_1n}}{m_1n}\right)^a\right)-W_{m_1n,i}\right|<\delta.
 \end{equation}
 Thus for $n\gg 0$ and $0<i<m_1$,
 $$
 \begin{array}{lll}
 c-\frac{\epsilon}{2}&<& ((-\overline A)^b\cdot(\frac{U_{m_1n}}{m_1n})^a)=
 \left| ((-\overline A)^b\cdot(\frac{U_{m_1n}}{m_1n})^a) -W_{m_1n,i}+W_{m_1n,i}\right| \\
 &\le& \left| ((-\overline A)^b\cdot(\frac{U_{m_1n}}{m_1n})^a) -W_{m_1n,i}\right|+\left|W_{m_1n,i}\right| \\
 &<& \frac{\epsilon}{2}+((-\overline A)^b\cdot (\frac{U_{m_1n+i}}{m_1n+i})^a)
  \end{array}
 $$
 where the last inequality is by  (\ref{eq12}) and (\ref{eq19}). Thus
 \begin{equation}\label{eq18}
 c-\epsilon<((-\overline A)^b\cdot (\frac{U_{m_1n+i}}{m_1n+i})^a)<c
 \end{equation}
 for $n\gg 0$ and $0< i<m_1$. Theorem \ref{Theorem2} now follows from (\ref{eq13}) and (\ref{eq18}).
\end{proof}

\subsection{Limits of Amao multiplicities}\label{SecAmlim}

\begin{Theorem}\label{Theorem3}
Let $R$ be a $d$-dimensional local ring which is essentially of finite type over a field $k$ and $I\subset R$ be an ideal. Then the limit
$$
a\epsilon(I)=\lim_{n\rightarrow\infty} \frac{a(I^n,(I^n)^{\rm sat})}{n^d}\in \RR
$$
exists as a sum of limits of Segre products.
\end{Theorem}

\begin{proof} Let notation be as in Section \ref{SecComp}. By Theorem \ref{Theorem1}, 
$$
a(I^n,(I^n)^{\rm sat})=\frac{1}{d!}(\sum_{i=1}^ds_i(n))
$$
where 
$$
s_i(n)=((-D_n-E_n)^{i-1}\cdot (-E_n)^{d-i}\cdot D_n).
$$
Thus 
$$
s^i(n)=( (-n\overline A)^{i-1}\cdot(-n\overline A+D_n)^{d-i}\cdot D_n))
$$
by equations (\ref{eq6}), (\ref{eq7}) and (\ref{eq17}).
$$
( (-n\overline A)^{i-1}\cdot (-n\overline A+D_n)^{d-i}\cdot D_n))
=((-n\overline A)^{i-1}\cdot(-n\overline A+D_n)^{d-i+1})-((-n\overline A)^{i}\cdot(n\overline A+D_n)^{d-i}).
$$
The limits
$$
\lim_{n\rightarrow\infty}\frac{((-n\overline A)^{i-1}\cdot (-n\overline A+D_n)^{d-i+1})}{n^d}\in \RR
$$
and
$$
\lim_{n\rightarrow\infty}\frac{((-n\overline A)^{i}\cdot(-n\overline A+D_n)^{d-i})}{n^d}\in \RR
$$
exist by Theorem \ref{Theorem2}. Thus $\lim_{n\rightarrow\infty}\frac{s_i(n)}{n^d}\in \RR$ exists, establishing Theorem \ref{Theorem3}.

\end{proof}


\begin{thebibliography}{0}


\bibitem{Am} J. Amao, On a certain Hilbert Polynomial, J. Lond. Math. Soc. 14 (1976), 13 - 20.


\bibitem{BH} W. Bruns and J. Herzog, Cohen-Macaulay rings, revised edition, Cambridge Univ. Press, 1998.



\bibitem{C2} S.D. Cutkosky, Asymptotic multiplicities of graded families of ideals and linear series, Adv. Math. 264 (2014), 55 - 113.


\bibitem{C4} S.D. Cutkosky, A general volume = multiplicity formula, Acta Math Vietnam 40 (2015), 139 - 147.
\bibitem{C5} S.D. Cutkosky, Multiplicities of graded families of ideals on Noetherian local rings, arXiv:2603.06844.

\bibitem{CHST} S.D. Cutkosky, T. Ha, H. Srinivasan and E. Theodorescu, Asymptotic behavior of length of local cohomology, Canad. J. Math. 57 (2005), 1178 - 1192.

\bibitem{CL} S.D. Cutkosky and S. Landsittel, Epsilon multiplicity is a limit of Amao multiplicities,  J. Algebra Appl. 24 (2025).

\bibitem{CS} S.D. Cutkosky and P. Sarkar, Epsilon Multiplicity and Analytic Spread of Filtrations,  Illinois J. Math. 68 (2024), 189 - 210. 

\bibitem{SD} Sudipta Das, A Volume = Multiplicity formula for $p$-families of ideals, Proc. AMS. 151 (2023), 4153 - 4161.

\bibitem{DM} Sudipta Das and Cheng Meng, Asymptotic colengths for families of ideals: an analytic approach, J. Algebra 701 (2026), 822 - 889.

\bibitem{sD} Suprajo Das, Epsilon multiplicity for graded algebras, J. Pure Appl. Algebra 225, 2021

\bibitem{DDRV} Suprajo Das, S. Dubey, S. Roy and  J. Verma, Computing epsilon multiplicities in graded algebras, J. Pure Appl. Algebra 229 (2025).

\bibitem{DRT} Suprajo Das, S. Roy and V. Trivedi, Density functions for epsilon multiplicity and families of ideals, J. London Math. Soc. 111 (2025).

\bibitem{DRT2} Suprajo Das, S. Roy and V. Trivedi, Numerical characterization for integral dependence of graded ideals, Int. Math. Res. Not. 2026.



\bibitem{ELS} L. Ein, R. Lazarsfeld and K. Smith, Uniform Approximation of Abhyankar valuation ideals in smooth function fields, Amer. J. Math. 125 (2003), 409-440.

 \bibitem{GG} T. Gaffney and R. Gassler, Segre numbers and hypersurface singularities, J. Algebraic Geom. 8 (1999), 695–736.
 
\bibitem{HHRT} M. Herrmann, E. Hyry, J. Ribbe and Z. Tang, Reduction numbers and multiplicities of multigraded structures, J. Algebra 197 (1997), 311 - 341.


\bibitem{KR} D. Kirby and D. Rees, Multiplicities in graded rings I: the general theory, Contemp. Math 159 (1994), 209 - 267.

\bibitem{KT} S. Kleiman and A. Thorup, A geometric theory of the Buchsbaum-Rim multiplicity, J. Alg. 167 (1994), 168 - 231.

\bibitem{L} S. Landsittel, Some formulas for epsilon multiplicity in local rings,  Com. in Algebra 54 (2026), 1962 - 1975.

\bibitem{Laz1} R. Lazarsfeld, Positivity in Algebraic Geometry I, Springer, 2000.

\bibitem{LM} R. Lazarsfeld and M. Musta\c{t}\u{a}, Convex bodies associated to linear series, Ann. Sci. Ec. Norm. Super 42 (2009) 783 - 835.

\bibitem{Ma} H. Matsumura, Commutative Algebra, second edition, Benjamin Cummings, Reading Massachusetts, 1980.

\bibitem{M1} M. Musta\c{t}\u{a}, On multiplicities of graded sequences of ideals, J. Algebra 256 (2002), 229-249.


\bibitem{R} D. Rees, Amao's theorem and reduction criteria, J. Lond. Math. Soc. 32 (1985), 404 - 410.

\bibitem{PS} P. Sarkar, Multiplicities of weakly graded families of ideals, Bull. London Math. Soc. 57 (2025), 2354 - 2371.



\bibitem{HS} I. Swanson and C. Huneke, integral closure of ideals, rings and modules, Cambridge Univ. Press.

\bibitem{T1} B. Teissier,   Cycles \'evanescents, sections planes et conditions de Whitney, Ast\'erisque 7-8 (1973) 285 - 362.                                                                    

\bibitem{T2} B. Teissier, Cycles \'evanescents et resolution simultan\'ee, I et II, S\'eminaire sur les singularities 1976 - 1977, Springer Lecture Notes in Mathematics 777, 1980, 71 - 146.


\bibitem{Th} A. Thorup, Rational equivalence theory on arbitrary Noetherian schemes, In Enumerative Geometry, Sebastian
Xambó-Descamps Ed, LNM 1436, Springer, 1990.

\bibitem{TV} N.V. Trung and J.K. Verma, Hilbert functions of multigraded algebras, mixed multiplicities of ideals and their applications, J. Commutative Algebra 2 (2010), 515 - 565.

\bibitem{UV} B. Ulrich and J. Validashti, Numerical criteria for integral dependence, Math. Proc. Camb. Phil. Soc. 151 (2011), 95-102.



\end{thebibliography}
\end{document}